\documentclass[11pt,reqno]{amsart} 

\pdfoutput=1

\usepackage{amssymb,amsfonts,amsmath,epsfig,yhmath,color}

 \allowdisplaybreaks
\numberwithin{equation}{section}

\font\script=rsfs10 at 11pt
\def\H{{\mbox{\script H}\,\,}}
\def\fn{$\partial^+ B^+_\delta$}
\def\fd{$\partial^- B^-$}
\def\C{\mathcal C}
\def\N{\mathbb N}
\def\R{\mathbb R}
\def\S{\mathbb S}
\def\J{\mathfrak J}
\def\eps{\varepsilon}
\def\step#1#2{\par\noindent{\underline{\it Step~#1.}}\emph{ #2}\\}
\def\proofof#1{\begin{proof}[Proof of~#1]}
\def\freccia#1{\xrightarrow[\ #1]{}}
\def\XXint#1#2#3{{\setbox0=\hbox{$#1{#2#3}{\int}$} \vcenter{\vspace{-1pt}\hbox{$#2#3$}}\kern-.5\wd0}}
\def\Xint#1{\mathchoice {\XXint\displaystyle\textstyle{#1}}{\XXint\textstyle\scriptstyle{#1}}{\XXint\scriptstyle\scriptscriptstyle{#1}}{\XXint\scriptscriptstyle\scriptscriptstyle{#1}}\!\int}
\def\intmed{\Xint{-}}

\newtheorem{theorem}{Theorem}[section]
\newtheorem{lemma}[theorem]{Lemma}
\newtheorem{definition}[theorem]{Definition}
\newtheorem{prop}[theorem]{Proposition}

\newtheorem{remark}[theorem]{Remark}

\begin{document}

\title[Existence of isoperimetric sets with densities ``converging from below'' on $\R^N$]{Existence of isoperimetric sets with densities ``converging from below'' on $\R^N$}

\author{Guido De Philippis}\address{Institut f\"ur Mathematik Universit\"at Z\"urich, Winterthurerstr. 190, CH-8057 Z\"urich (Switzerland)}\email{guido.dephilippis@math.uzh.ch}

\author{Giovanni Franzina}\address{Department Mathematik, University of Erlangen, Cauerstr. 11, 91058 Erlangen (Germany)}\email{franzina@math.fau.de}
\author{Aldo Pratelli}\address{Department Mathematik, University of Erlangen, Cauerstr. 11, 91058 Erlangen (Germany)}\email{pratelli@math.fau.de}

\begin{abstract}
In this paper, we consider the isoperimetric problem in the space $\R^N$ with density. Our result states that, if the density $f$ is l.s.c. and converges to a limit $a>0$ at infinity, being $f\leq a$ far from the origin, then isoperimetric sets exist for all volumes. Several known results or counterexamples show that the present result is essentially sharp. The special case of our result for radial and increasing densities posively answers a conjecture made in~\cite{PM13}.
\end{abstract}

\maketitle

\section{Introduction}
In this paper we are interested in the isoperimetric problem with a weight. This means that we are given a positive l.s.c. function $f:\R^N\to \R^+$, usually called ``density'', and we measure volume and perimeter of a generic subset $E$ of
$\R^N$ as
\begin{align*}
|E|_f:= \H_f^N(E)= \int_E f(x) \, d \H^N\,, &&
P_f(E) : = \H_f^{N-1}(\partial^M E)=\int_{\partial^M E} f(x)\, d\H^{N-1}(x)\,,
\end{align*}
where the \emph{essential boundary} of $E$ (which coincides with the usual topological boundary as soon as $E$ is regular) is defined as
\[
\partial^M E  = \bigg\{ x\in \R: \liminf_{r\searrow0} \frac{\H^N(E\cap B_r(x))}{\omega_Nr^N}<1  \ \hbox{and} \ 
\limsup_{r\searrow 0}	\frac{\H^N(E\cap B_r(x))}{\omega_Nr^N}>0 \bigg\}\,,
\]
$B_r(x)$ stands for the ball of radius $r$ centered at $x$, and $\omega_N$ is the euclidean volume of a ball of radius 1. This problem, and many specific cases, have been extensively studied in the last decades and have many important applications; a short (highly non complete) list of some related papers is~\cite{Alm,BigReg,M1,MJ,newM1,BCMR,CMV,CRS,D,PM13,CP12}.\par\medskip

The first interesting question in this setting is of course the existence of \emph{isoperimetric sets}, that are sets $E$ with the property that $P_f(E)= \J(|E|_f)$ where, for any $V\geq 0$,
\begin{equation*}
\J(V):=\inf\big\{P_f(F):\, |F|_f=V\big\}\,.
\end{equation*}
Depending on the assumptions on $f$, the answer to this question may be trivial or extremely complicate.

Let us start with a very simple, yet fundamental, observation. Fix a volume $V>0$ and let $\{E_i\}$ be an \emph{isoperimetric sequence} of volume $V$: this means that $|E_i|_f=V$ for every $i\in\N$, and $P_f(E_i) \to \J(V)$. Thus, possibly up to a subsequence, the sets $E_i$ converge to some set $E$ in the $L^1_{\rm loc}$ sense. As a consequence, standard lower semi-continuity results in BV ensure that $P_f(E)\leq \liminf P_f(E_i)=\J(V)$ (at least, for instance, if $f>0$\dots); therefore, if actually $|E|_f=V$, then obviously $E$ is an isoperimetric set. Unfortunately, this simple observation is not sufficient, in general, to show the existence of isoperimetric sets, because there is no general reason why the volume of $E$ should be exactly $V$ (while it is obviously at most $V$).\par

A second remark to be done is the following: if the volume of the whole space $\R^N$ is finite, then in the argument above it becomes obvious that $|E|_f=V$; basically, the mass cannot vanish to infinity. Hence, in this case isoperimetric sets exist for all volumes.\par

Let us then consider the more general (and interesting) problem when $f\not\in L^1(\R^N)$. In this case, by the different scaling properties of volume and perimeter, roughly speaking we can say that ``isoperimetric sets like small density''. Let us be a little bit more precise: one can immediately check that, if two different balls $B_1$ and $B_2$ lie in two regions where the density is constantly $d_1$ resp. $d_2$, and if $|B_1|_f=|B_2|_f$, then $P_f(B_1)< P_f(B_2)$ as soon as $d_1<d_2$. More in general, all the simplest examples show that isoperimetric sets tend to privilege the zones where the density is lower, and it is very reasonable to expect that this behaviour is quite general. Of course, this argument does not predict anything in situations where the density varies quickly (for instance, it would be very convenient for a set to lie where the density is large if at the same time the boundary stays where the density is small!), but nevertheless having this ``general rule'' in mind may help a lot.\par

With the aid of the above observation, let us now come back again to the question of the existence of isoperimetric sets. If the density $f$ converges to $0$ at infinity, one has to expect that isoperimetric sets do not exist (remember that we are assuming $\R^N$ to have infinite volume, otherwise the existence is always true). Indeed, in general a sequence of sets of given volume minimizing the perimeter diverges to infinity, to reach the zones with lowest density, and then actually the infimum of the perimeter for sets of any given volume is generally $0$.\par

On the contrary, if the density $f$ blows up at infinity, one has to expect isoperimetric sets to exist: indeed, in this case the sequences minimizing the perimeter should remain bounded in order not to go where the density is high, and hence the limit of a minimizing sequence $\{E_i\}$ as above should have volume $V$, and then it would be an isoperimetric set. A complete answer to this question has been already given in~\cite{PM13}: if the density is also radial, then isoperimetric sets exist for every volume, as expected (Theorem~3.3 in~\cite{PM13}), but if the density is not radial, then the existence might fail (Proposition~5.3 in~\cite{PM13}), contrary to the intuition.\par

Let us then pass to consider the case when the density, at infinity, is neither converging to $0$ nor diverging. Again, it is very simple to observe that existence generally fails if the density is decreasing, at least definitively; analogously, it is easy to build both examples of existence and of non-existence for oscillating densities (that is, densities for which the $\liminf$ and the $\limsup$, at infinity, are different). Summarizing, for what concerns the existence problem, the only interesting case left is when the density has a finite limit at infinity, and it is converging to that limit from below. This leads us to the following definition.

\begin{definition}\label{defconvdens}
We say that the l.s.c. function $f:\R^N\to \R$ is \emph{converging from below} if there exists $0<a <+\infty$ such that $f(x) \to a$ when $|x|\to \infty$, and $f(x) \leq a$ for $|x|$ big enough.
\end{definition}

Basically, the observations above tell that, for functions $f$ which are not converging densities, there is in general no interesting open question about the existence issue. Indeed, as explained above, in each of these cases it is already known whether isoperimetric sets exist for all volumes or not. Conversely, for some special cases of densities converging from below, the existence problem has been already discussed. In particular, combining the results of~\cite{PM13} and~\cite{CP12}, the existence of isoperimetric sets follows for densities which are continuous and converging from below and which satisfy some technical assumptions, for instance it is enough that $f$ is superharmonic, or that $f$ is radial and for every $c>0$ there is some $R\gg 1$ for which $f(R)\leq a-e^{-cR}$. Moreover, in~\cite{PM13} it was conjectured that isoperimetric sets exist for all volumes if the density is radial and increasing.\par

In this paper we are able to prove the existence result for any density converging from below (this is even stronger than the above-mentioned conjecture); as explained above, this result is sharp.

\begin{theorem}\label{main}
Let $f\in L^1_{\rm loc}(\R^N)$ be a density converging from below. Then, isoperimetric sets exist for every volume.
\end{theorem}

\section{General results about isoperimetric sets\label{sec:genfact}}

In this section we present a couple of general facts about existence and boundedness of isoperimetric sets.\par

As already briefly described in the Introduction, let us fix some $V>0$ and an \emph{isoperimetric sequence} of volume $V$, that is, a sequence of sets $E_j\subseteq \R^N$ such that $|E_j|_f=V$ for any $j$, and $P_f(E_j)\to \J(V)$ for $j\to \infty$. As already observed, if (a subsequence of) $\{E_j\}$ converges in $L^1_{\rm loc}$ to a set $E$, then by lower semicontinuity $P_f(E)\leq \J(V)$, and $|E|_f\leq V$; thus, the set $E$ is automatically isoperimetric of volume $V$ if $|E|_f=V$. However, it is always true that $E$ is isoperimetric \emph{for its own volume}. We stress that this fact is widely known, but we prefer to give the proof to keep the presentation self-contained, and also because we could not find in the literature any proof which works in such a generality. After this lemma, we will show that if there was loss of mass at infinity (that is, if $|E|_f< V$), then $E$ is necessarily bounded.
\begin{lemma}\label{isopF}
Assume that $f\in L^1_{\rm loc}(\R^N)$ and that $f$ is locally bounded from above far enough from the origin. Let $\{E_j\}$ be an isoperimetric sequence of volume $V$ converging in $L^1_{\rm loc}$ to some set $E$. Then, $E$ is an isoperimetric set for the volume $|E|_f$. If in addition $f$ is converging to some $a>0$, then
\begin{equation}\label{thm1.2}
\J(V) = P_f(E) + N(\omega_N a)^{\frac 1N} (V-|E|_f )^{\frac{N-1}N}\,.
\end{equation}
\end{lemma}
\begin{proof}
Let us start proving that $E$ is isoperimetric. As we already observed, $P_f(E)\leq \J(V)$ and $|E|_f\leq V$; as a consequence, if $|E|_f=V$ it is clear that $E$ is isoperimetric, and on the other hand if $|E|_f=0$ then the empty set $E$ is still clearly isoperimetric for the volume $0$. As a consequence, we can assume without loss of generality that $0< |E|_f < V$.\par

Suppose now that the claim is false, and let then $F_1$ be a set satisfying
\begin{align*}
|F_1|_f=|E|_f\,, && \eta:=\frac{P_f(E)-P_f(F_1)}6 >0\,.
\end{align*}
Select now $x\in \R^N$ being a point of density $1$ in $F_1$ and a Lebesgue point for $f$ with $f(x)>0$: such a point exists, in particular $\H^N_f$-a.e. point of $F_1$ can be taken. The assumptions on $x$ ensure that, for every radius $\bar r$ small enough,
\begin{equation}\label{buonvol}
\frac 12 \,\omega_N f(x) \bar r^N \leq |B_{\bar r}(x) \cap F_1|_f 
\leq |B_{\bar r}(x)|_f  \leq 2 \omega_N f(x) {\bar r}^N\,,
\end{equation}
and in turn this implies that there exist arbitrarily small radii $r$ (not necessarily all those small enough) such that
\begin{equation}\label{buonper}
\H^{N-1}_f \big(\partial B_r(x)\big) \leq 2 N\omega_N f(x) r^{N-1}\,.
\end{equation}
Indeed, if the last inequality were false for every $0<r<\bar r$, then by integrating we would get that~(\ref{buonvol}) is false.\par Analogously, let $y$ be a point of density $0$ for $F_1$ which is Lebesgue for $f$ with $f(y)>0$ (the existence of such a point requires that $f\notin L^1(\R^N)$, which on the other hand is surely true because $|E|_f<V$). Since we can find such a point arbitrarily far from the origin (and far from $x$), by assumption it is admissible to assume that $f\leq M$ in a small neighborhood of $y$. As a consequence, there exists some radius $\bar\rho>0$ such that, for every $0<\rho<\bar\rho$,
\begin{align}\label{propdelta0}
\big|B_\rho(y)\setminus F_1 \big|_f \geq \frac{f(y)}2\, \omega_N \rho^N\,, && \H^{N-1}_f\big(\partial B_\rho(y)\big) \leq M N \omega_N \rho^{N-1}\,.
\end{align}
Let us now fix a constant $\delta>0$ such that (up to possibly decrease $\bar\rho$)
\begin{align}\label{propdelta}
\delta < \eta\,, && \frac{f(y)}2\, \omega_N \bar\rho^N > \delta\,, && M N \omega_N \bar\rho^{N-1} < \eta\,.
\end{align}
We claim the existence of some set $F\subseteq \R^N$ and of a big constant $R>0$ (in particular, much bigger than both $|x|$ and $|y|$) such that
\begin{align}\label{defF}
F \subseteq B_R\,, && P_f(F) < P_f(E) - 5\eta\,, &&  0< \delta':= |E|_f -|F|_f< \frac \delta 2\,,
\end{align}
writing for brevity $B_R=B_R(0)$. To show this, it is useful to consider two possible cases. If $F_1$ is bounded, we define $F=F_1\setminus B_r(x)$ for some $r$ very small such that both~(\ref{buonvol}) and~(\ref{buonper}) hold true. Then, the inclusion $F\subseteq B_R$ is true for every $R$ big enough, and the two inequalities in~(\ref{defF}) immediately follow by~(\ref{buonvol}), (\ref{buonper}) and the definition of $\eta$ as soon as $r$ is sufficiently small. Instead, if $F_1$ is not bounded, then we define $F=F_1\cap B_R$ for a big constant $R$: of course the inclusion $F\subseteq B_R$ is automatically satisfied, and the inequality about $\delta'$ is also true for every $R$ big enough, say $R>R_0$. Concerning the inequality on $P_f(F)$, if it were false for every $R>R_0$, then for every $R>R_0$ it would be
\[
\H^{N-1}_f\big(F_1\cap \partial B_R\big) \geq \eta\,,
\]
and then by integrating we would get
\[
V > |F_1|_f \geq |F_1\setminus B_{R_0}|_f = \int_{R_0}^{+\infty} \H^{N-1}_f \big(F_1 \cap \partial B_R\big) = +\infty\,,
\]
and the contradiction shows the existence of some suitable $R$, thus the existence of $F$ satisfying~(\ref{defF}) is proved.\par
We can now select some $R'>R$ such that
\begin{align}\label{hypR}
|E\setminus B_{R'}|_f < \frac{\delta'}2\,,  &&  \H^{N-1}_f(\partial E\cap B_{R'}) > P_f(E) - \eta\,.
\end{align}
Since $E_j\cap B_{R'}$ (resp., $E_j \cap B_{R'+1}$) converges in the $L^1$ sense to $E\cap B_{R'}$ (resp., $E \cap B_{R'+1}$), for every $j$ big enough we have
\begin{gather}
|E|_f -\delta' < |E_j\cap B_{R'}|_f \leq |E_j\cap B_{R'+1}|_f < |E|_f+ \delta'\,, \label{R'bigvol} \\
\H^{N-1}_f(\partial E \cap B_{R'})\leq \H^{N-1}_f(\partial E_j \cap B_{R'})+\eta\,. \label{R'bigper}
\end{gather}
Arguing as above, by~(\ref{R'bigvol}) we have
\[
\delta>  2\delta' \geq \Big|E_j \cap \big(B_{R'+1}\setminus B_{R'}\big)\Big|_f
= \int_{R'}^{R'+1} \H^{N-1}_f (E_j\cap \partial B_t)\, dt\,,
\]
so we can find some $R_j\in (R',R'+1)$ such that, also recalling~(\ref{propdelta}),
\begin{equation}\label{choiceR_j}
\H^{N-1}_f (E_j\cap \partial B_{R_j}) <\delta < \eta\,.
\end{equation}
Observe that, since $|E_j|=V$ by definition, (\ref{R'bigvol}) implies
\[
V - |E|_f -\delta' < |E_j\setminus B_{R_j}|_f < V - |E|_f +\delta'\,.
\]
As a consequence, calling $G_j = F \cup \big( E_j \setminus B_{R_j}\big)$ and also recalling~(\ref{defF}), (\ref{hypR}), (\ref{R'bigper}) and~(\ref{choiceR_j}), we can estimate the volume of $G_j$ by
\begin{equation}\label{volGj}
|G_j|_f = |F|_f + |E_j \setminus B_{R_j}|_f = |E|_f - \delta' + |E_j \setminus B_{R_j}|_f \in (V- \delta, V)\,,
\end{equation}
and the perimeter of $G_j$ by
\begin{equation}\label{perGj}\begin{split}
P_f(G_j) &= P_f(F) + P_f(E_j\setminus B_{R_j})\\
&< P_f (E) - 5\eta + \H^{N-1}_f(\partial E_j\setminus B_{R_j}) + \H^{N-1}_f (E_j \cap \partial B_{R_j})\\
&<\H^{N-1}_f(\partial E \cap B_{R'})  + \H^{N-1}_f(\partial E_j\setminus B_{R_j}) - 3\eta
\leq P_f (E_j)- 2\eta\,.
\end{split}\end{equation}
Finally, we define the competitor $\widetilde E_j = G_j \cup B_{\rho_j}(y)$, where $\rho_j<\bar\rho$ is the constant such that $|\widetilde E_j|_f = V$ --this is possible by~(\ref{volGj}), (\ref{propdelta0}), and~(\ref{propdelta}). Applying then again~(\ref{propdelta0}) and~(\ref{propdelta}), from~(\ref{perGj}) we deduce
\[
P_f(\widetilde E_j) < P_f(E_j) - \eta
\]
for every $j$ big enough, and this gives the desired contradiction with the fact that the sequence $E_j$ was isoperimetric. This finally shows that $E$ is an isoperimetric set for the volume $|E|_f$.\par\medskip

Let us now pass to the second part of the proof, namely, we assume that $f$ is converging to some $a>0$ (not necessarily from below), and we aim to prove~(\ref{thm1.2}). Notice that we can assume without loss of generality that $|E|_f < V$, since otherwise~(\ref{thm1.2}) is a direct consequence of the fact that $E$ is isoperimetric.\par

Arguing as in the first part of the proof, for every $\eps>0$ we can find a very big $R$ such that, calling $F=E\cap B_R$, it is
\begin{align*}
|F|_f \geq |E|_f - \eps\,, && P_f(F) \leq P_f(E) +\eps\,.
\end{align*}
Let then $B$ a ball with volume $|B|_f = V - |F|_f$: if we take this ball far enough from the origin, then $B\cap F=\emptyset$, thus $|G|_f=V$ being $G=F\cup B$; moreover, again up to take the ball far enough, we have $a-\eps\leq f \leq a+\eps$ on the whole $B$. As a consequence, calling $r$ the radius of $B$, we have
\[
V-|E|_f + \eps\geq 
V-|F|_f = |B|_f \geq (a-\eps) \omega_N r^N\,,
\]
from which we get
\[\begin{split}
\J(V) &\leq P_f(G) = P_f(F) + P_f(B) 
\leq P_f(E) + \eps + (a+\eps) N \omega_N r^{N-1}\\
&\leq P_f(E) + \eps +  \frac{a+\eps}{(a-\eps)^{\frac{N-1}N}}\,N\omega_N^{\frac 1N}\,\Big(V-|E|_f + \eps\Big)^{\frac{N-1}N}\,,
\end{split}\]
which in turn implies the first inequality in~(\ref{thm1.2}) by letting $\eps\to 0$.\par

To show the other inequality, consider again the isoperimetric sequence $\{E_j\}$; for any given $\eps>0$, exactly as in the first part we can find an arbitrarily big $R$ so that $a-\eps\leq f\leq a+\eps$ out of $B_R$ and
\begin{align*}
|E\cap B_R|_f \geq |E|_f - \eps\,, && P_f (E\setminus B_R) \leq \eps\,.
\end{align*}
For every $j\gg 1$, then, we can find some $R_j \in (R,R+1)$ so that
\begin{align*}
|E_j\cap B_{R_j}|_f \leq |E|_f +\eps \,, && 
\H^{N-1}_f (E_j \cap \partial B_{R_j}) \leq 2\eps\,, && 
P_f(E) \leq P_f(E_j\cap B_{R_j}) +2\eps\,.
\end{align*}
Since $a-\eps\leq f\leq a+\eps$ out of $B_R$, we deduce
\[\begin{split}
P_f (E_j\setminus B_{R_j}) &\geq (a-\eps) P_{\rm eucl} (E_j\setminus B_{R_j}) \geq (a-\eps) N\omega_N^{\frac 1N}|E_j\setminus B_{R_j}|_{\rm eucl}^{\frac{N-1}N}\\
&\geq \frac{a-\eps}{(a+\eps)^{\frac{N-1}N}}\,N\omega_N^{\frac 1N}|E_j\setminus B_{R_j}|_f^{\frac{N-1}N}
\geq \frac{a-\eps}{(a+\eps)^{\frac{N-1}N}}\,N\omega_N^{\frac 1N}\Big(V-|E|_f -\eps\Big)^{\frac{N-1}N}\,,
\end{split}\]
which in turn gives
\[\begin{split}
P_f(E_j) &= P_f(E_j\cap B_{R_j} ) + P_f(E_j\setminus B_{R_j}) - 2\H^{N-1}_f (E_j\cap \partial B_{R_j})\\
&\geq P_f(E) - 6\eps +\frac{a-\eps}{(a+\eps)^{\frac{N-1}N}}\,N\omega_N^{\frac 1N}\Big(V-|E|_f -\eps\Big)^{\frac{N-1}N}\,.
\end{split}\]
Since $P_f(E_j)\to \J(V)$ for $j\to\infty$, by sending $\eps\to 0$ in the last estimate yields the second inequality in~(\ref{thm1.2}), thus the proof is concluded.
\end{proof}

\begin{remark}{\rm
Actually, the claim of Lemma~\ref{isopF} can be proved even with weaker assumptions; more precisely, one could apply the results of~\cite{CP12} to extend the validity to the more general case when $f$ is ``essentially bounded'' in the sense of~\cite{CP12}.}
\end{remark}

The second result that we present is a clever observation, which we owe to the courtesy of Frank Morgan, and which shows that whenever a density converges to a limit $a>0$ (not necessarily from below), then if an isoperimetric sequence is losing mass at infinity the remaining limiting set --which is isoperimetric thanks to Lemma~\ref{isopF}-- is bounded.

\begin{lemma}[Morgan]\label{lemma:morgan}
Let the density $f$ converge to some $a>0$, and let the isoperimetric sequence $\{E_j\}$ of volume $V$ converge in $L^1_{\rm loc}$ to a set $E$ with $|E|_f<V$. Then, $E$ is bounded.
\end{lemma}
\begin{proof}
Assume that $|E|_f<V$. Then, for every $t>0$ define
\[
m(t) = |E\setminus B_t|_f = \int_t^\infty \H_f^{N-1} (E\cap \partial B_\sigma)\,d\sigma\,.
\]
For every $t$, we can select a ball $B$ of volume $V-|E|_f+ m(t)$ far away from the origin, in order to have no intersection with $E\cap B_t$; thus, the set $(E\cap B_t) \cup B$ has precisely volume $V$, hence $\J(V) \leq P_f (E\cap B_t) + P_f(B)$. Since the ball $B$ can be taken arbitrarily far from the origin, thus in a region where $f$ is arbitrarily close to $a$, exactly as in the second part of the proof of Lemma~\ref{isopF} we deduce
\[
\J(V) \leq P_f (E \cap B_t) + N (a\omega_N)^{\frac 1N} \big(V-|E|_f+ m(t)\big)^{\frac{N-1}N}\,.
\]
Recalling that $|E|_f<V$ and comparing the last inequality with~(\ref{thm1.2}), we obtain
\[
P_f(E) \leq P_f (E\cap B_t) + C m(t)
\]
for some strictly positive constant $C$. Notice now that
\[
P_f(E) = P_f(E\cap B_t) + P_f(E\setminus B_t) - 2 \H^{N-1}_f(E\cap \partial B_t)
= P_f(E\cap B_t) + P_f(E\setminus B_t) + 2 m'(t)\,,
\]
and in turn by the (Euclidean) isoperimetric inequality if $t\gg 1$ we have
\[
P_f(E\setminus B_t) \geq (a-\eps) P_{\rm eucl}(E\setminus B_t)
\geq (a-\eps) N\omega_N^{\frac 1N} |E\setminus B_t|^{\frac{N-1}N}_{\rm eucl}
\geq \frac{a-\eps}{(a+\eps)^{\frac{N-1}N}}\, N\omega_N^{\frac 1N} m(t)^{\frac{N-1}N}\,.
\]
Putting everything together, we get
\[
Cm(t) \geq 2 m'(t) + \frac 1{C_1} m(t)^{\frac{N-1}N}
\]
for some other constant $C_1>0$. And in turn, if $t\gg 1$ then $m(t)\ll 1$, thus the last estimate implies
\[
m(t) \leq C_2 \big(- m'(t)\big)^{\frac N{N-1}}\,.
\]
Finally, it is well known that a positive decreasing function $m$ which satisfies the above differential inequality vanishes in a finite time. Hence, $m(t)=0$ for $t$ big enough, and this means precisely that $E$ is bounded.
\end{proof}

\section{Proof of the main result\label{proofmain}}

This section is devoted to show the main result of the paper, namely, Theorem~\ref{main}. Our overall strategy is quite simple, and already essentially contained in~\cite{PM13}. The idea is to take an isoperimetric sequence of volume $V$, and to consider a limiting set $E$ (up to a subsequence, this is always possible); if $|E|_f=V$, then there is nothing to prove because, as we already saw several times, the set $E$ is already the desired isoperimetric set of volume $V$. Instead, if $|E|_f<V$, we know by Lemma~\ref{isopF} that $E$ is an isoperimetric set for volume $|E|_f$, and by Lemma~\ref{lemma:morgan} that $E$ is bounded. Moreover, formula~(\ref{thm1.2}) says that an isoperimetric set of volume $V$ can be found as the union of $E$ and a ``ball at infinity'' with volume $V-|E|_f$. By ``ball at infinity'' we mean an hypothetical ball where the density is constantly $a$: such a ball needs not really to exist, but a sequence of balls of correct volume which escape at infinity will have a perimeter which converges to that of this ``ball at infinity''. In other words, a sequence of sets done by the union of $E$ and a ball escaping at infinity is isoperimetric thanks to~(\ref{thm1.2}). Our strategy is then simple: we look for a set $B$, far away from the origin, which is \emph{better} than a ball at infinity, that is, which has the same volume and less perimeter than it. Since $E$ is bounded (this is a crucial point, from which the importance of Lemma~\ref{lemma:morgan}) the sets $E$ and $B$ have no intersection, thus the union of $E$ with $B$ is isoperimetric. As one can see, the only thing to do is to find a set of given volume, arbitrarily far from the origin, which is ``better'' than a ball at infinity.\par

First of all, let us express in a useful way the fact of being better than a ball at infinity, by means of the following definition.
\begin{definition}
We say that the set $E\subseteq \R^N$ of finite volume has \emph{mean density} $\rho$ if
\[
P_f(E) = N(\omega_N \rho)^\frac 1N |E|_f^\frac{N-1}N\,.
\]
\end{definition}
The meaning of this definition is evident: $\rho$ is the unique number such that, if we endowe $\R^N$ with the constant density $\rho$, then balls of volume $|E|_f$ have perimeter $P_f(E)$. The convenience of this notion is also clear: being ``better than a ball at infinity'' simply means having mean density less than $a$.\par

We can then continue our description of the proof of Theorem~\ref{main}: we are left to find a set of volume $V-|E|_f$ arbitrarily far from the origin and having mean density at most $a$. Since we want to find isoperimetric set for any volume $V$, and we cannot know a priori how much $|E|_f$ is, we need to find sets of mean density less than $a$ \emph{of any volume} and \emph{arbitrarily far from the origin}. Actually, by a trivial rescaling argument, we can assume that $a=1$ and reduce ourselves to search for a set of volume $\omega_N$. Since $f$ is converging to $1$ and we must work very far from the origin, everything will be very close to the Euclidean case, hence a set of volume $\omega_N$ and mean density less than $1$ (or, equivalently, with perimeter less than $N \omega_N$) must be extremely close to a ball of radius $1$. The first big step in our proof will then be to find a ball of radius $1$ arbitrarily far from the origin, and with mean density less than $1$.\par

Surprisingly enough, this will by no means conclude the proof, due to a seemingly minor problem: indeed, since $f$ converges to $1$ from below, the ball of radius $1$ that we have found does not have exactly volume $\omega_N$, but only a bit less. And, the far from the origin the ball is, the smaller this gap will be, but still positive. Notice that at this point we cannot again rely on a rescaling argument: we have already rescaled in order to reduce ourselves to the case of volume $\omega_N$, but then any other volume will not solve the problem (in principle, it could be that there are sets of mean density less than $1$ only for all the rational volumes, and for no irrational one\dots). Hence, the second big step in our proof will be to slightly modify the ball found in the first big step, in such a way that the volume increases up to exactly $\omega_N$, while the mean density remains smaller than $1$. At that point, the proof will be concluded. It is to be mentioned that the proof of this second fact is more delicate than the one of the first!\par

Let us now state precisely the claims of the two big steps, and then give the formal proof of Theorem~\ref{main} --which is more or less exactly what we have just described informally. Then, we will conclude the paper with two sections, which are devoted to present the proof of the two big claims.

\begin{prop}\label{farball}
Let $f$ be a density converging from below to $1$, and set $g=1-f$. Then, for every $\eps>0$ there exists a ball $B$ with radius $1$ and arbitrarily far from the origin such that
\[
P_g(B)\geq (N-\eps)|B|_g\,.
\]
\end{prop}
\begin{prop}\label{prop:setmdgen}
Let $f$ be a density converging from below to $1$. Then, there exists a set $E$ with volume $\omega_N$ and mean density smaller than $1$ arbitrarily far from the origin.
\end{prop}

\proofof{Theorem~\ref{main}}
Let $\{E_j\}$ be an isoperimetric sequence of volume $V$, and let $E$ be the $L^1_{\rm loc}$ limit of a suitable subsequence. If $|E|_f=V$ then the proof is already concluded. Otherwise, we know that $E$ is bounded by Lemma~\ref{lemma:morgan} and that~(\ref{thm1.2}) holds. Up to a rescaling, we can assume that $f$ converges from below to $1$, and that $V-|E|_f=\omega_N$. By Proposition~\ref{prop:setmdgen} we can find a set $F$ not intersecting $E$ with volume $\omega_N$ and mean density less than $1$, which means $P_f(F)\leq N\omega_N$. The set $E\cup F$ has then volume $V$, and by~(\ref{thm1.2}) we obtain $P(E\cup F)\leq \J(V)$, which means that $E\cup F$ is an isoperimetric set.
\end{proof}

\subsection{Proof of Proposition~\ref{farball}}

This section is devoted to the proof of Proposition~\ref{farball}. Before presenting it, it is convenient to show a couple of technical lemmas.

\begin{lemma}\label{lm:sgnint}
Let $g: (0,\infty)\to[0,\infty)$ and $\alpha:(-1,1)\to\R$ be $L^1$ functions such that
\begin{align}\label{alphacond}
\lim_{t\to \infty} g(t)=0\,, &&
\int_{-1}^1\alpha(t)\,dt = 0\,, &&
\int_{-1}^\sigma \alpha(t) \,dt >0 \quad \forall \,\sigma \in (-1,1)\,.
\end{align}
Then there exists an arbitrarily large $R$ such that
\[
\int_{-1}^1\alpha(t)g(t+R)\,dt\geq 0\,,
\]
with strict inequality unless $g(t)=0$ for all $t$ big enough.
\end{lemma}
\begin{proof}
If the claim were false, then for every choice of $R',\,R''$ with $R''\geq R'+2$ one had
\[\begin{split}
0&>\int_{R'}^{R''}\int_{-1}^1\alpha(t)g(t+R)\,dt\,dR
=\int_{R'-1}^{R'+1}g(s)\int_{-1}^{s-R'}\alpha(t)\,dt\,ds
+\int_{R''-1}^{R''+1}g(s)\int_{s-R''}^{1}\alpha(t)\,dt\,ds\\
&=A(R')+B(R'')\,,
\end{split}\]
where there is no integral over $(R'+1,R''-1)$ because it cancels thanks to~(\ref{alphacond}). The conditions on $\alpha$ and $g$ also ensure that $A(R')\geq 0\geq B(R'')$ for every $R',\, R''$. Suppose now that for some arbitrarily large $R'$ one has $A(R')>0$; we can then fix $R'$ and send $R''\to \infty$: since $g\to 0$, we get $B(R'')\to 0$, and then there is some $R''\gg 1$ such that $A(R')+B(R'')>0$, against the above inequality. As a consequence, it must be $A(R')=0$ for every $R'$ big enough, and in turn this means that $g$ is definitively zero, hence any $R$ big enough satisfies the claim.
\end{proof}

\begin{lemma}\label{lm:sgnintbis}
Let $g: (0,\infty)\to [0,\infty)$ and $\beta:(-1,1)\to\R$ be $L^1$ functions such that $g$ and $\alpha(t)=\int_{-1}^t\beta(\sigma)\,d\sigma$ satisfy condition~\eqref{alphacond}, and $\alpha(1)=0$. Then, there exists an arbitrarily large $R$ such that
\begin{equation}\label{dimbeta}
\int_{-1}^1\beta(t)g(t+R)\,dt\geq 0\,,
\end{equation}
with strict inequality unless $g(t)=0$ for all $t$ big enough.
\end{lemma}
\begin{proof}
The proof is analogous to the one of Lemma~\ref{lm:sgnint} above. Take $R'\gg 1$ and assume that the conclusion fails for every $R\geq R'$: then, for every $R''> R' +2$ we have
\[
0>\int_{R'}^{R''}\int_{-1}^{1} \beta(t)g(t+R)\,dt\,dR
=\int_{R'-1}^{R'+1}g(s)\int_{-1}^{s-R'}\beta(t)\,dt\,ds+\int_{R''-1}^{R''+1}g(s)\int_{s-R''}^{1}\beta(t)\,dt\,ds\,.
\]
Exactly as before, since the last term in the right goes to $0$ when $R''\to\infty$, we find a contradiction as soon as the first term in the right is strictly positive. In other words, the proof is concluded as soon as we find some $R'$ such that
\[
0<\int_{R'-1}^{R'+1}g(s)\int_{-1}^{s-R'}\beta(t)\,dt\,ds
=\int_{R'-1}^{R'+1} g(s) \alpha(s-R')\,ds
=\int_{-1}^1 \alpha(t)g(t+R')\,dt\,.
\]
And in turn, the existence of such an $R'$ is ensured by Lemma~\ref{lm:sgnint} since $\alpha$ satisfies condition~(\ref{alphacond}), unless $g$ is definitively zero. And in this latter case, of course any $R$ big enough would satisfy the required condition.
\end{proof}

We are now in position to prove Proposition~\ref{farball}.

\proofof{Proposition~\ref{farball}}
For simplicity, we split the proof in two steps: first we show that one can always reduce himself to the case of a radial density, and then we prove the claim for this case.
\step{I}{Reduction to radial case.}
Let us assume that the claim holds for any radial density, and let $f$ be not necessarily radial. Define then the density $\tilde f$ as the radial average of $f$, namely,
\begin{equation}\label{deftildef}
\tilde f(x) = \intmed_{\partial B_{|x|}} f(y)\, d\H^{N-1}(y)\,.
\end{equation}
Of course, then $\tilde g=1-\tilde f$ is also the radial average of $g$. Since the claim holds for the radial density $\tilde f$, for any $\eps>0$ we can find a ball $B$ satisfying $P_{\tilde g}(B) \geq (N-\eps) |B|_{\tilde g}$. Let us then call $B^\theta$, for $\theta\in\S^{N-1}$, the ball having the same distance from the origin as $B$, and which is rotated of an angle $\theta$: all the different balls $B^\theta$ are equivalent for the density $\tilde f$, but not for the original density $f$. Observe now that by definition
\begin{align*}
P_{\tilde g} (B) = \intmed_{\S^{N-1}} P_g (B^\theta) \, d\H^{N-1}(\theta)\,, &&
|B|_{\tilde g}= \intmed_{\S^{N-1}} |B^\theta|_g \, d\H^{N-1}(\theta)\,,
\end{align*}
and then of course there exists some $\theta\in\S^{N-1}$ such that $P_g(B^\theta)\geq (N-\eps) |B^\theta|_g$.

\step{II}{Proof of the radial case.}
Thanks to Step~I we can assume without loss of generality that $f$ is radial. For a ball $B_R$ having radius $1$ and center at a distance $R$ from the origin, we can then calculate perimeter and volume by integrating over the radial layers, that is, we have
\begin{align}\label{exactform}
P_g(B_R) = \int_{-1}^1 \varphi_R(t) g(t+R) \, dt\,, &&
|B_R|_g = \int_{-1}^1 \psi_R(t) g(t+R) \, dt\,,
\end{align}
where $\varphi_R(t)$ and $\psi_R(t)$ can be calculated by Fubini Theorem and co-area formula. Actually, it is not important to write down the exact formula, while it is immediate to observe that (basically, since the layers become flat in the limit) the following uniform limits hold
\begin{align}\label{limitstilde}
\frac{\varphi_R(t)}{\widetilde\varphi(t)} \freccia{R\to\infty} 1\,, &&
\frac{\psi_R(t)}{\widetilde\psi(t)} \freccia{R\to\infty} 1\,,
\end{align}
being the limit functions simply
\begin{align*}
\widetilde\varphi(t) = (N-1)\omega_{N-1} (1-t^2)^{\frac{N-3}2}\,, &&
\widetilde\psi(t) = \omega_{N-1} (1-t^2)^{\frac{N-1}2}\,.
\end{align*}
As a consequence, we can work with the approximated functions $\widetilde\varphi$ and $\widetilde\psi$ in place of $\varphi$ and $\psi$: more precisely, we call ``approximated'' perimeter and volume of $B_R$ the functions $\widetilde P_g(B_R)$ and $\widetilde V_g(B)$ obtained by substituting $\varphi$ and $\psi$ in~(\ref{exactform}) with $\widetilde\varphi$ and $\widetilde\psi$. The claim will be then automatically obtained, thanks to~(\ref{limitstilde}), if we can find an arbitrarily large $R$ such that
\[
\widetilde P_g(B_R)\geq N \widetilde{V}_g(B_R)\,.
\]
We can now define $\beta:(-1,1)\to \R$ as $\beta(t)=\tilde\varphi(t)-N\tilde\psi(t)$, so that we are reduced to find an arbitrarily large $R$ such that~(\ref{dimbeta}) holds. It is elementary to check that the assumptions of Lemma~\ref{lm:sgnintbis} are satisfied: one can either do the simple calculations, or just observe that $\alpha(t)$ coincides with the perimeter minus $N$ times the volume of the portion of the unit ball centered at the origin whose first coordinate is between $-1$ and $t$, so that all the conditions to check become trivial. Therefore, the existence of the searched $R$ directly comes from Lemma~\ref{lm:sgnintbis}, and the proof is completed.
\end{proof}

\subsection{Proof of Proposition~\ref{prop:setmdgen}}
This last section is entirely devoted to give the proof of Proposition~\ref{prop:setmdgen}, which is again divided in some steps. For convenience of the reader, in Steps~I and~II we start with two particular cases, namely, when $f$ is non-decreasing along the half-lines starting at the origin, and when $f$ is radial: even though these two particular cases are not really needed for the proof, the argument is similar to the general one but works more easily, so this helps to understand the general case.

\proofof{Proposition~\ref{prop:setmdgen}}
Let us fix $\eps\ll 1$: thanks to Proposition~\ref{farball}, there is a ball $B=B_R^{\bar\theta}$ of radius $1$ and centered at the point $R\bar\theta$, with some arbitrarily large $R$ and some $\bar\theta\in \S^{N-1}$, which satisfies $P_g(B)\geq (N-\eps) |B|_g$. Since $f\leq 1$ on $B$, we have $|B|_f\leq \omega_N$: if $|B|_f=\omega_N$ we are already done, because $P_f(B)\leq P_{\rm eucl}(B)=N\omega_N$, and this automatically implies that the mean density of $B$ is less than $1$. Let us then suppose that $|B|_f <\omega_N$, or equivalently that $|B|_g>0$, and let us try to enlarge $B$ so to reach volume $\omega_N$, but still having mean density less than $1$. We will do this in some steps.

\step{I}{The case of non-decreasing densities.}
Let us start with the case when $f$ is a ``non-decreasing density'': this means that, for every $\theta\in \S^{N-1}$, the function $t\mapsto f(t\theta)$ is non-decreasing, at least for large $t$.\par

In this case, let us define a new set $E$ as follows. First of all, we decompose $B=B_l \cup B_r$, where $B_l$ and $B_r$ are the ``left'' and the ``right'' part of the ball $B_R^{\bar\theta}$: formally, a point $x\in B$ is said to belong to $B_l$ or $B_r$ if $x\cdot \bar\theta$ is smaller or bigger than $R$ respectively. Then, for any small $\delta$, we call $B_{l,\delta}$ the half ball centered at $(R-\delta)\bar\theta$ with radius $(R-\delta)/R$, and $C_\delta$ the cylinder of radius $1$ and height $\delta$ whose axis is the segment connecting $(R-\delta)\bar\theta$ and $R\bar\theta$; finally, we let $E_\delta=B_r \cup B_{l,\delta} \cup C_\delta$, see Figure~\ref{Fig:sets}, left. Since $f$ is converging to $1$, and $R$ can be taken arbitrarily big, we have
\[
|E_\delta|_f  - |B|_f \geq (1-\eps ) \omega_{N-1} \delta \,;
\]
as a consequence, by continuity we can fix $\bar\delta$ such that $E=E_{\bar \delta}$ has exactly volume $\omega_N$, and we have
\begin{equation}\label{estibardelta}
\bar\delta \leq (1+2 \eps) \,\frac{|B|_g}{\omega_{N-1}}\,.
\end{equation}
Thanks to the assumption that $f$ is non-decreasing, we know that
\begin{equation}\label{diminbound}
\H^{N-1}_f (\partial^l B_{l,\delta}) \leq \H^{N-1}_f (\partial^l B_l)\,,
\end{equation}
where we call $\partial^l B_{l,\delta}$ and $\partial^l B_\delta$ the ``left parts'' of the boundaries, that is,
\begin{align*}
\partial^l B_l = \Big\{y \in \partial B_l :\, y \cdot \bar\theta \leq R \Big\}\,, &&
\partial^l B_{l,\delta} = \Big\{y \in \partial B_{l,\delta} :\, y \cdot \bar\theta \leq R-\delta \Big\}\,.
\end{align*}
As a consequence, using again that $f\leq 1$ and that $R$ can be taken arbitrarily big, thanks to~(\ref{estibardelta}) and~(\ref{diminbound}) we can evaluate
\[\begin{split}
P_f(E) &\leq P_f(B) +(N-1+\eps)\omega_{N-1}\bar\delta 
\leq N\omega_N - P_g(B)  +(N-1+\eps) (1+2\eps) |B|_g\\
&\leq N\omega_N - (N-\eps) |B|_g  +(N-1+\eps) (1+2\eps) |B|_g < N\omega_N\,.
\end{split}\]
Summarizing, we have built a set $E$ arbitrarily far from the origin, with volume exactly $\omega_N$, and perimeter less than $N\omega_N$, thus mean density less than $1$. The proof is then concluded for this case.

\begin{figure}[thbp]
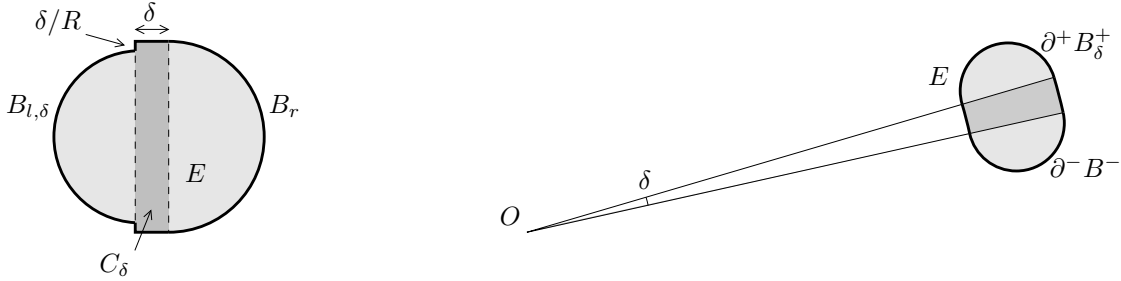
\caption{The sets $E$ of Step~I (left) and of Step~II (right). The half-balls $B_r$ and $B_{l,\delta}$, as well as the half-balls $B^-$ and $B^+_\delta$, are light shaded; the cylinder $C_\delta$, as well as the region $E\setminus (B^-\cup B^+_\delta)$, is dark shaded.}\label{Fig:sets}
\end{figure}

\step{II}{The case of radial densities.}
Let us now assume that the density is radial. In this case, we cannot use the same argument as in the previous step, because there would be no way to extend the validity of~(\ref{diminbound}). Nevertheless, we can use a similar idea to enlarge the ball $B$, namely, instead of translating half of the ball $B$ we rotate it. More formally, let us take an hyperplane passing through the origin and the center of the ball $B_R^{\bar\theta}$, and let us call $B^\pm$ the two corresponding half-balls in which $B_R^{\bar\theta}$ is subdivided. Let us then consider the circle contained in $\S^{N-1}$ which contains the direction $\bar\theta$ and the direction orthogonal to the hyperplane, and for any small $\sigma>0$ call $\rho_\sigma$ the rotation of an angle $\sigma$ with respect to this circle. Then, let us call $B^+_\sigma=\rho_\sigma(B^+)$ and finally let $E_\delta$ be the union of $B^-$ with all the half-balls $B^+_\sigma$ for $0<\sigma<\delta$, as in Figure~\ref{Fig:sets}, right. As in the previous step, since $f$ is converging to $1$ we can evaluate the difference of the volumes as
\[
|E_\delta|_f  - |B|_f \geq \omega_{N-1}(R-1)(1-\eps) \delta \,,
\]
then we can again select $\bar\delta$ such that $E=E_{\bar\delta}$ has volume exactly $\omega_N$ and we have
\begin{equation}\label{giamer}
\bar\delta \leq  (1+2 \eps)\,\frac{|B|_g}{\omega_{N-1}(R-1)}\,.
\end{equation}
This time, the radial assumption on $f$ gives
\[
\H^{N-1}_f(\partial^+ B^+_\delta)=\H^{N-1}_f (\partial^+ B^+)\,,
\]
where we call $\partial^+ B^+_\delta$ and $\partial^+ B^+$ the ``upper'' parts of the boundaries in the obvious sense. And finally, almost exactly as in last step we can evaluate the perimeter of $E$ as
\[\begin{split}
P_f(E) &\leq P_f(B) + (N-1)\omega_{N-1} (R+1) \bar\delta
\leq N\omega_N - P_g(B) + (N-1)(1+2\eps)\,\frac{R+1}{R-1}\,|B|_g\\
&\leq N\omega_N - (N-\eps)|B|_g + (N-1)(1+2\eps)\,\frac{R+1}{R-1}\,|B|_g< N\omega_N\,,
\end{split}\]
where the last inequality again is true if we have chosen $\eps\ll 1$ and then $R\gg 1$. Thus, the set $E$ has volume $\omega_N$ and mean density less than $1$, and the proof is obtained also in this case.

\step{III}{The general case in dimension $2$.}
Let us now treat the case of a general density $f$. For simplicity of notations we assume now to be in the two-dimensional situation $N=2$, and in the next step we will generalize our argument to any dimension.\par
As in the proof of Proposition~\ref{farball}, let us call $\tilde f$ the radial average of $f$ according to~(\ref{deftildef}), and $\tilde g=1-\tilde f$ the radial average of $g$. Proposition~\ref{farball} provides then us with a ball $B_R$, of radius $1$ and distance $R\gg 1$ from the origin, such that
\begin{equation}\label{choiceR}
P_{\tilde g}(B_R)\geq (N-\eps) |B_R|_{\tilde g}\,.
\end{equation}
For any $\theta \in \S^1$, as usual, we call then $B_R^\theta$ the ball of radius $1$ centered at $R\theta$. Let us now argue as in Step~II: we call $B_R^{\theta,\pm}$ (resp., $\partial^\pm B_R^\theta$) the two half-balls (resp., half-circles) made by the points of $B_R^\theta$ (resp., $\partial B^\theta_R$) having direction bigger or smaller than $\theta$; thus, for any small $\delta>0$, we define $E_\delta^\theta$ the union of $B^{\theta,-}_R$ with all the half-balls $B^{\theta+\sigma,+}_R$ for $0<\sigma<\delta$. Since the sets $E^\theta_\delta$ are increasing for $\delta$ increasing, if $R\gg 1$ there is a unique $\bar\delta=\bar\delta(\theta)$ such that $|E^\theta_{\bar\delta}|_f=\omega_N$, and exactly as in Step~II we have the 
estimate~(\ref{giamer}) for $\bar\delta$, which for $R$ big enough (since $f\to 1$ and then $g\to 0$) implies
\begin{equation}\label{perdopo}
\bar\delta(\theta) \leq \frac{(1+3\eps)|B_R^\theta|_g}{\omega_{N-1}(R-1)}\,.
\end{equation}
Let us then define the function $\tau:\S^1\to\S^1$ as $\tau(\theta)=\theta+\bar\delta(\theta)$, and notice that by construction this is a strictly increasing bijection of $\S^1$ onto itself, with $\tau(\theta)>\theta$ (if $\tau(\theta)=\theta$ then the ball $B^\theta_R$ has already volume $\omega_N$, and in this case there is nothing to prove, as already observed). Let us now fix a generic $\theta\in\S^1$, and let $\eta\ll \tau(\theta)-\theta$: if we call
\begin{align*}
A=\Big(\bigcup\nolimits_{0<\sigma<\eta} B_R^{\theta+\sigma} \Big)\setminus B_R^{\theta+\eta}\,, &&
B=\Big(\bigcup\nolimits_{0<\sigma<\eta} B_R^{\tau(\theta+\sigma)} \Big)\setminus B_R^{\tau(\theta)}\,,
\end{align*}
then, since
\begin{align*}
\big|E^\theta_{\bar\delta(\theta)}\big|_f = \omega_N = \big|E^{\theta+\eta}_{\bar\delta(\theta+\eta)}\big|_f\,, &&
E^{\theta+\eta}_{\bar\delta(\theta+\eta)}=\big(E^\theta_{\bar\delta(\theta)}\cup B\big) \setminus A\,,
\end{align*}
one has $|A|_g =|B|_g$. On the other hand, one clearly has
\[
\frac{|B|_{\rm eucl}}{|A|_{\rm eucl}}= \frac{\tau(\theta+\eta)-\tau(\theta)}\eta\,,
\]
Up to take $R$ big enough, we can assume without loss of generality that $1-\eps\leq f \leq 1$ for points having distance at least $R-1$ from the origin, and this yields
\[
1-\eps \leq \frac{\tau(\theta+\eta) - \tau(\eta)}\eta \leq \frac 1{1-\eps}\,.
\]
As an immediate consequence, we get that the function $\tau$ is bi-Lipschitz and $1-\eps\leq\tau'\leq (1-\eps)^{-1}$. Let us now observe that, by construction, all the sets $E^\theta=E^\theta_{\tau(\theta)-\theta}$ have exactly volume $\omega_N$: we want then to find some $\bar\theta\in \S^1$ such that $P_f(E^{\bar\theta})\leq N\omega_N$, so $E^{\bar\theta}$ has mean density less than $1$ and we are done. Now, since a simple change of variables gives
\[
\intmed_{\S^1} \H^{N-1}_g \big(\partial^+ B^\theta_R\big)\,d\theta = \intmed_{\S^1} \H^{N-1}_g \big(\partial^+ B^{\tau(\nu)}_R\big) \tau'(\nu) \,d\nu
\leq \frac 1{1-\eps}\ \intmed_{\S^1} \H^{N-1}_g \big(\partial^+ B^{\tau(\theta)}_R\big) \,d\theta\,,
\]
we can readily evaluate by~(\ref{choiceR})
\[\begin{split}
0&\leq P_{\tilde g}(B_R) -(N-\eps) |B_R|_{\tilde g} = \intmed_{\S^1} P_g(B_R^\theta) - (N-\eps)|B_R^\theta|_g\,d\theta \\
&= \intmed_{\S^1} \H^{N-1}_g \big(\partial^+ B^\theta_R\big)\,d\theta+\intmed_{\S^1} \H^{N-1}_g \big(\partial^- B^\theta_R\big)\,d\theta-(N-\eps)\intmed_{\S^1} |B_R^\theta|_g\,d\theta
\\
&\leq \intmed_{\S^1} \frac 1{1-\eps}\H^{N-1}_g\big(\partial^+ B^{\tau(\theta)}_R\cup\partial^-B^\theta_R\big)-(N-\eps)|B_R^\theta|_g\,d\theta\,,
\end{split}\]
and hence get the existence of some $\bar\theta\in \S^1$ such that
\[
\H^{N-1}_g\big(\partial^+ B^{\tau(\bar\theta)}_R\cup\partial^-B^{\bar\theta}_R\big)\geq (1-\eps) (N-\eps) |B_R^{\bar\theta}|_g\,.
\]
Thanks to~(\ref{perdopo}), we have then
\[\begin{split}
P_f\big(E^{\bar\theta}\big)&= \H^{N-1}_f\big(\partial^+ B^{\tau(\bar\theta)}_R\cup\partial^-B^{\bar\theta}_R\big)
+\H^{N-1}_f\Big(\partial E^{\bar\theta} \setminus \big(\partial^+ B^{\tau(\bar\theta)}_R\cup \partial^- B^{\bar\theta}_R\big)\Big)\\
&\leq N\omega_N - \H^{N-1}_g\big(\partial^+ B^{\tau(\bar\theta)}_R\cup\partial^-B^{\bar\theta}_R\big)
+ (N-1)\omega_{N-1} \bar\delta(\bar\theta) (R+1)\\
&\leq N\omega_N - (1-\eps)(N-\eps) |B_R^{\bar\theta}|_g+ (N-1)(1+3\eps)|B_R^{\bar\theta}|_g < N\omega_N\,,
\end{split}\]
where the last inequality holds as soon as $\eps$ was chosen small enough at the beginning. The set $E^{\bar\theta}$ is then as searched and this step is done.

\step{IV}{The general case.}
We are now ready to conclude the proof in the general case. We start noticing that in the argument of Step~III the assumption $N=2$ was used only to work with $\S^1$, hence to get the validity of~(\ref{choiceR}). More precisely, let us assume that there exists some arbitrarily large $R$ and some circle $\C\approx \S^1$ in $\S^{N-1}$ such that the estimate
\begin{equation}\label{general}
\intmed_\C P_g (B_R^\theta) \, d\H^1(\theta) \geq (N-\eps) \intmed_\C |B_R^\theta|_g\,d\H^1(\theta)
\end{equation}
holds true. Then, we can repeat \emph{verbatim} the proof of Step~III, we get the existence of some $\bar\theta\in \C$ such that the set $E_R^{\bar\theta}$ has volume $\omega_N$ and mean density less than $1$, and the proof is concluded. Hence, we are left to find some $R$ and some circle $\C$ so that~(\ref{general}) holds; notice that, if $N=2$, then it must be $\C=\S^1$ and~(\ref{general}) reduces to~(\ref{choiceR}), which in turn holds for some arbitrarily large $R$ thanks to Proposition~\ref{farball}.\par

Let us then consider the case of dimension $N=3$. By Proposition~\ref{farball} we can take $R\gg 1$ such that~(\ref{choiceR}) holds true; for any $\theta\in \S^2$, then, we can call $\C_\theta$ the circle in $\S^2$ which is orthogonal to $\theta$, and observe that by homogeneity
\begin{align*}
P_{\tilde g}(B_R) = \intmed_{\S^2} \intmed_{\C_\theta} P_g (B_R^\sigma)\, d\H^1(\sigma) \, d\H^2(\theta)\,, &&
|B_R|_{\tilde g} = \intmed_{\S^2} \intmed_{\C_\theta} |B_R^\sigma|_g\, d\H^1(\sigma) \, d\H^2(\theta)\,,
\end{align*}
so thanks to~(\ref{choiceR}) we get the existence of a circle $\C=\C_{\bar\theta}$ for which~(\ref{general}) holds true: the proof is then concluded also in dimension $N=3$.\par

Notice that the argument above can be rephrased as follows: if there exists some sphere $\mathcal S\approx \S^2 \subseteq \S^{N-1}$ such that the average estimate~(\ref{general}) holds with $\mathcal S$ in place of $\C$ (and in turn in dimension $N=3$ this reduces to~(\ref{choiceR}) and hence holds), then the proof is concluded. As a consequence, the claim follows also in dimension $N=4$, arguing exactly as above with the spheres $\mathcal S_\theta\approx \S^2$ orthogonal to any $\theta\in \S^3$, and the obvious induction argument gives then the thesis for any dimension.
\end{proof}

\begin{remark}{\rm
Notice that, in the proof of Proposition~\ref{prop:setmdgen}, we have actually found a set which has mean density \emph{strictly less} than $1$, unless $g\equiv 0$ on some ball of radius $1$. On the other hand, as clearly appears from the proof of Theorem~\ref{main}, it is impossible to find such a set if some isoperimetric sequence is losing mass at infinity: indeed, otherwise the argument of Theorem~\ref{main} would give a set with perimeter strictly less than the infimum. There are then only two possibilities: either there are balls where $f\equiv 1$ arbitrarily far from the origin, or no isoperimetric sequence can lose mass at infinity.\par
In particular, our proof shows that no isoperimetric sequence can lose mass at infinity if $f<1$ out of some big ball.}
\end{remark}

\section*{Acknowledgment}
The work of the three authors was supported through the ERC St.G. 258685. We wish also to thank Michele Marini and Frank Morgan for useful discussions and comments.

\end{document}